\documentclass{article}
\usepackage[a4paper, total={7in, 10in}]{geometry}
\usepackage[T1]{fontenc}
\usepackage{graphicx}
\usepackage{mathtools}
\usepackage{amsmath,amsthm,amssymb,mathrsfs,amstext, titlesec,enumitem, comment, graphicx, color, xcolor, stmaryrd,mathabx}
\usepackage{thmtools}
\usepackage{xpatch}
\usepackage{tikz-cd}
\usepackage{nameref}
\usepackage{hyperref}
\makeatletter
\def\Ddots{\mathinner{\mkern1mu\raise\p@
\vbox{\kern7\p@\hbox{.}}\mkern2mu
\raise4\p@\hbox{.}\mkern2mu\raise7\p@\hbox{.}\mkern1mu}}
\makeatother
\newtheorem{theorem}{Theorem}[section]
\newtheorem{corollary}[theorem]{Corollary}
\newtheorem{lemma}[theorem]{Lemma}
\newtheorem{question}[theorem]{Question}
\theoremstyle{definition}
\newtheorem{definition}[theorem]{Definition}

\title{\textbf{A concept of largeness of combinatorially rich sets}}

\date{}
\author{Pintu Debnath
        \footnote{Department of Mathematics,
                 Basirhat College,
                  Basirhat-743412, North 24th Parganas, West Bengal, India. \hfill\break
                  {\tt pintumath!989@gmail.com}}
}

\begin{document}

\maketitle
\begin{abstract}
In \cite[Proposition 8.21 Page-169]{F}  Using the methods of topological dynamics,  H. Furstenberg introduced the notion of central set and proved the famous Central Sets Theorem. Later, in \cite{DHS}, D. De, H. Hindman and D. Struss established a strong Central Sets Theorem, where they introduced the notion of $J$-set. Like $J$-set, in \cite{BG} V. Bergelson and D. Glasscock introduced the notion of combinatorially rich set ( $CR$-set).
Let $u,v\in\mathbb{N}$ and $A$ be a $u\times v$ matrix with rational entries. In \cite{HS23}
N. Hindman and D. Strauss established
 that whenever $B$ is a piecewise syndetic set (resp. $J$-set) in
$\mathbb{Z}$, $\left\{ \vec{x}\in\mathbb{Z}^{v}:A\vec{x}\in B^{u}\right\} $
is a piecewise syndetic set ( resp. $J$-set) in $\mathbb{Z}^{v}$. In this article, we prove the same result for $CR$-sets using an equivalent definition of $CR$-set in \cite{HHST} by  N. Hindman, H. Hosseini, D. Strauss and M. Tootkaboni. 
\end{abstract}

\section{Introduction}

We start by presenting  a brief review of the algebraic structure of the Stone-\v{C}ech
	compactification of a semigroup $\left(S,+\right)$, not necessarily commutative with the discrete topology, to present some results in this article.
	The set $\{\overline{A}:A\subset S\}$ is a basis for the closed sets
	of $\beta S$. The operation `$+$' on $S$ can be extended to
	the Stone-\v{C}ech compactification $\beta S$ of $S$ so that $(\beta S,+)$
	is a compact right topological semigroup (meaning that for each    $p\in\beta$ S the function $\rho_{p}\left(q\right):\beta S\rightarrow\beta S$ defined by $\rho_{p}\left(q\right)=q+ p$ 
	is continuous) with $S$ contained in its topological center (meaning
	that for any $x\in S$, the function $\lambda_{x}:\beta S\rightarrow\beta S$
	defined by $\lambda_{x}(q)=x+q$ is continuous). This is a famous
	Theorem due to Ellis that if $S$ is a compact right topological semigroup
	then the set of idempotents $E\left(S\right)\neq\emptyset$. A nonempty
	subset $I$ of a semigroup $T$ is called a $\textit{left ideal}$
	of $S$ if $T+I\subset I$, a $\textit{right ideal}$ if $I+T\subset I$,
	and a $\textit{two sided ideal}$ (or simply an $\textit{ideal}$)
	if it is both a left and right ideal. A $\textit{minimal left ideal}$
	is the left ideal that does not contain any proper left ideal. Similarly,
	we can define $\textit{minimal right ideal}$ and $\textit{smallest ideal}$.
	
	Any compact Hausdorff right topological semigroup $T$ has the smallest
	two sided ideal
	
	$$
	\begin{aligned}
		K(T) & =  \bigcup\{L:L\text{ is a minimal left ideal of }T\}\\
		&=  \bigcup\{R:R\text{ is a minimal right ideal of }T\}.
	\end{aligned}$$

	Given a minimal left ideal $L$ and a minimal right ideal $R$, $L\cap R$
	is a group, and in particular contains an idempotent. If $p$ and
	$q$ are idempotents in $T$ we write $p\leq q$ if and only if $p+q=q+p=p$.
	An idempotent is minimal with respect to this relation if and only
	if it is a member of the smallest ideal $K(T)$ of $T$. Given $p,q\in\beta S$
	and $A\subseteq S$, $A\in p+ q$ if and only if the set $\{x\in S:-x+A\in q\}\in p$,
	where $-x+A=\{y\in S:x+ y\in A\}$. See \cite{HS12} for
	an elementary introduction to the algebra of $\beta S$ and for any
	unfamiliar details.

 \begin{definition}
   Let $\left(S,+\right)$ be a commutative semigroup and let $A\subseteq S$.
   \begin{itemize}
 \item[(a)](\textbf{Piecewise syndetic set}) A is piecewise syndetic if and only if there exists $G\in\mathcal{P}_{f}\left(S\right)$
such that for every $F\in\mathcal{P}_{f}\left(S\right)$, there is
some $x\in S$ such that $F+x\subseteq\cup_{t\in G}(-t+A)$. Equivalently $A$ is a piecewise syndetic if and only if $\overline{A}\cap K\left(\beta S\right)\neq\emptyset$.
\item[(b)] (\textbf{$J$-set}) A is $J$-set if and only if for every $F\in\mathcal{P}_{f}\left({}^{\mathbb{N}}S\right)$,
there exist $a\in S$ and $H\in\mathcal{P}_{f}\left(\mathbb{N}\right)$
such that for each $f\in F$, $a+\sum_{n\in H}f(n)\in A$.
   \end{itemize}

\end{definition}

Let $J\left(S\right)=\left\{p\in\beta S:\left(\forall\in p\right) \left(A\text{ is a } J\text{ set }\right) \right\}$. By \cite[Lemma 14.14.5, Page-345]{HS12} and \cite[Theorem 3.20 Page-63]{HS12} $J\left(S\right)\neq\emptyset$. It is shown in \cite[Theorem 14.14.4 Page-345]{HS12} that the set $J\left(S\right)$  is a compact two sided ideal of $\beta S$

\begin{definition}
    Let $\left(S,+\right)$ be a commutative semigroup and let $A\subseteq S$.
    \begin{itemize}
        \item[(a)](\textbf{Central set}) $A$ is a central set if and only if there is an idempotent in $\overline{A}\cap\ K\left(\beta S\right)$.
        \item[(b)](\textbf{$C$-set}) $A$ is a $C$ set if and only if there is an idempotent in $\overline{A}\cap\ J\left(S\right)$.
    \end{itemize}
\end{definition}

\begin{definition}
    Let $u,v\in\mathbb{N}$, and let $A$ be a $u\times v$ matrix with rational entries.
    \begin{itemize}
        \item[(a)](\textbf{Image partition regular over $\mathbb{N}$}) The matrix $A$ is image partition regular over $\mathbb{N}$ ($IPR/\mathbb{N}$) if and only if, whenever $\mathbb{N}$ is finitely colored, there exists $\vec{x}\in\mathbb{N}^{v}$ such that the entries of $A\vec{x}$ are monocromatic.

        \item[(b)](\textbf{Image partition regular over $\mathbb{Z}$}) The matrix $A$ is image partition regular over $\mathbb{Z}$ ($IPR/\mathbb{Z}$) if and only if, whenever $\mathbb{Z}\setminus\left\{0\right\}$ is finitely colored, there exists $\vec{x}\in\left\{\mathbb{Z}\setminus\left\{0\right\}\right\}^{v}$ such that the entries of $A\vec{x}$ are monocromatic.
    \end{itemize}
\end{definition}

\begin{theorem}\label{imge in central}\cite[Theorem 2.10]{HLS}
    Let $u,v\in\mathbb{N}$ and let $A$ be a $u\times v$ matrix with entries from $\mathbb{Q}$. The following statements are equivalent.
    \begin{itemize}
        \item[(a)] $A$ is image partition regular over $\mathbb{N}$
        \item[(b)] For every central set $C$ in $\mathbb{N}$, there exists $\vec{x}\in\mathbb{N}^{v}$ such that $A\vec{x}\in C^{v}$.
        \item[(c)] For every central set $C$ in $\mathbb{N}$, $\left\{\vec{x}\in\mathbb{N}^{v}:A\vec{x}\in C^{v}\right\}$ is central in $\mathbb{N}^{v}$.
    \end{itemize}
\end{theorem}

\begin{theorem}\label{imge in C}\cite[Theorem 1.4]{HS20}
    Let $u,v\in\mathbb{N}$ and let $A$ be a $u\times v$ matrix with entries from $\mathbb{Q}$. The following statements are equivalent.
    \begin{itemize}
        \item[(a)] $A$ is image partition regular over $\mathbb{N}$.
        \item[(b)] For every C-set $C$ in $\mathbb{N}$, there exists $\vec{x}\in\mathbb{N}^{v}$ such that $A\vec{x}\in C^{v}$.
        \item[(c)] For every C-set  $C$ in $\mathbb{N}$, $\left\{\vec{x}\in\mathbb{N}^{v}:A\vec{x}\in C^{v}\right\}$ is a $C$-set in $\mathbb{N}^{v}$.
    \end{itemize}
\end{theorem}

There are many notions of largeness in a semigroup $S$ that have rich combinatorial properties . In \cite{H}, author has discussed about many of this notions. In \cite{BG}, V. Bergelson and D. Glasscock introduced a new notion of large sets in commutative semigroup $\left(S,+\right)$. They used matrix notation. Given a $r\times k$ matrix $M$ we denote by $m_{ij}$
the element in row $i$ and column $j$ of $M$.
\begin{definition}(\textbf{$CR$-set})
   Let $\left(S,+\right)$ be a commutative semigroup and let $A\subseteq S$.
Then $A$ is combinatorially rich set (denoted $CR$-set) if and only
if for each $k\in\mathbb{N}$, there exists $r\in\mathbb{N}$ such
that whenever $M$ is an $r\times k$ matrix with entries from $S$,
there exist $a\in S$ and nonempty $H\subseteq\left\{ 1,2,\ldots,r\right\} $
such that for each $j\in\left\{ 1,2,\ldots,k\right\} $, $$a+\sum_{t\in H}m_{t,j}\in A.$$  
\end{definition}
Let $CR\left(S\right)=\left\{p\in\beta S:\left(\forall\in p\right) \left(A\text{ is a } CR\text{ set }\right) \right\}$. By \cite[Lemma 2.14]{BG} and \cite[Theorem 3.20 Page-63]{HS12} $CR\left(S\right)\neq\emptyset$. As the family of $CR$-sets is translation invariant, then $CR\left(S\right)$ is a closed subsemigroup of $\beta S$.

\begin{definition}
   (\textbf{ Essential $CR$-set}) Let $\left(S,+\right)$ be a commutative semigroup and let $A\subseteq S$.
         $A$ is a $C$ set if and only if there is an idempotent in $\overline{A}\cap CR\left(S\right)$.
   
\end{definition}

It is know from \cite{BG} piecewise syndetic set $\implies$ $CR$-set and from definations of $CR$-set and $J$-set, we have   piecewise syndetic set $\implies$ $CR$-set $\implies$ $J$-set. And also  central set $\implies$ essential $CR$-set $\implies$ $C$-set. So from  Theorem  \ref{imge in central} snd Theorem \ref{imge in C}, we get the following theorem:

\begin{theorem}
    Let $u,v\in\mathbb{N}$ and let $A$ be a $u\times v$ matrix with entries from $\mathbb{Q}$. The following statements are equivalent.
    \begin{itemize}
        \item[(a)] $A$ is image partition regular over $\mathbb{N}$.
        \item[(b)] For every $CR$-set $C$ in $\mathbb{N}$, there exists $\vec{x}\in\mathbb{N}^{v}$ such that $A\vec{x}\in C^{v}$.
        
    \end{itemize}
\end{theorem}

We do not know the answer of the following naturally expected question:
\begin{question}
    Let $u,v\in\mathbb{N}$ and let $A$ be a $u\times v$ matrix with entries from $\mathbb{Q}$, which is image partition regular over $\mathbb{N}$.  Is for every $CR$-set  $C$ in $\mathbb{N}$, $\left\{\vec{x}\in\mathbb{N}^{v}:A\vec{x}\in C^{v}\right\}$  a $CR$-set in $\mathbb{N}^{v}$?
    
\end{question}

Given $n\in\mathbb{Z}$, we write $\overline{n}$ for the vector of the appropriate size all of whose entries are equal to $n$. For some certain class of matrices N. Hindman and D. Strauss  established the following two consecutive theorems for piecewise syndetic sets and $J$-sets.

\begin{theorem}\cite[Theorem 3.4]{HS23}
  Let $u,v\in\mathbb{N}$, let $A$ be a $u\times v$ matrix with rational
entries, and assume that for each $n\in\mathbb{Z}$, there exists
$\overrightarrow{z}\in\mathbb{Z}^{v}$ such that $A\overrightarrow{z}=\overline{n}\in\mathbb{Z}^{u}$.
If $B$ is piecewise syndetic subset of $\mathbb{Z}$, then $\left\{ \overrightarrow{x}\in\mathbb{Z}^{v}:A\overrightarrow{x}\in B^{u}\right\} $
is piecewise syndetic in $\mathbb{Z}^{v}$.   
\end{theorem}

\begin{theorem}\label{image in J}\cite[Theorem 3.7]{HS23}.
    Let $u,v\in\mathbb{N}$, let $A$ be a $u\times v$ matrix with rational
entries and let $B$ be a $J$-set in $\mathbb{Z}$. Assume that for
each $a\in\mathbb{Z}$, there exists $\overrightarrow{x}\in\mathbb{Z}^{v}$
such that $A\overrightarrow{x}=\overline{a}\in\mathbb{Z}^{u}$. Then
$\left\{ \overrightarrow{y}\in\mathbb{Z}^{v}:A\overrightarrow{y}\in B^{u}\right\} $
is a $J$-set in $\mathbb{Z}^{v}$. 
\end{theorem}

In the next section, we prove an analog result of the above theorem for $CR$-sets and essential $CR$-sets.

\section{Largeness of image in combinatorially rich sets}
As we promised to prove an analog result of Theorem \ref{image in J} for $CR$-set, we first established a variant of the following lemma:  

\begin{lemma}\cite[Lemma 3.3 ]{HS20}
    Let $F\in\mathcal{P}_{f}\left({}^{\mathbb{N}}\mathbb{Z}\right)$ and
let $d\in\mathbb{N}$. Then there exist a sequence $\left\{ K_{n}\right\} _{n=1}^{\infty}$
of finite sets of $\mathbb{N}$ and $k\in\mathbb{N}$ such that for
each $n\in\mathbb{N}$, $\max K_{n}<\min K_{n+1}$ so that for all $f\in F$
, $\sum_{t\in K_{n}}f\left(t\right)\in d\mathbb{Z}$.
\end{lemma}

If we pay  attention to the proof of Theorem \ref{image in J} in \cite[Theorem 3.7]{HS23}, we observe that the above lemma plays a crucial role  proving that theorem.

\begin{lemma}\label{main lemma}
    Let $F\in\mathcal{P}_{f}\left(\mathbb{N}_{\mathbb{Z}}\right)$ and
let $d\in\mathbb{N}$. Then there exist a sequence $\left\{ K_{n}\right\} _{n=1}^{\infty}$
of finite sets of $\mathbb{N}$ and $k\in\mathbb{N}$ such that for
each $n\in\mathbb{N}$, $\max K_{n}<\min K_{n+1}$ so that for all $f\in F$
,  $\sum_{t\in K_{n}}f(t)\in d\mathbb{Z}$  with $$K_{n}\subset\left\{ (n-1)k+1,(n-1)k+2,\ldots,nk\right\} $$ where
$k=d^{|F|}(d-1)+1$.
\end{lemma}

\begin{proof}
Let $|F|=m$, and $F=\left\{ \langle f_{i}\left(t\right)\rangle _{t=1}^{\infty}:i\in\left\{ 1,2,\ldots,m\right\} \right\} $. Let $k_{1}=d^{m}(d-1)+1$. Define  a partition  function $\chi_{1}$ on the set $\left\{ f_{1}(t):t\in\left\{ 1,2,\ldots,k\right\} \right\}$ by
$$
\chi_{1}:\left\{ f_{1}(t):t\in\left\{ 1,2,\ldots,k_{1}\right\} \right\} \rightarrow\left\{ 0,1,\ldots,d-1\right\} 
$$
 $\chi\left(f_{1}\left(t\right)\right)=r$, where $r$ is the remainder
, when we divide $f_{1}\left(t\right)$ by $d$.

By the pigeonhole principle, there exists
a subset $H_{1}\subset\left\{ 1,2,\ldots,k_{1}\right\} $, for which
$\chi\left(f_{1}\left(t\right)\right)=r$ for some $r\in\left\{ 0,1,\ldots,d-1\right\} $ and
for all $t\in H_{1}$ with $|H_{1}|=d^{m-1}(d-1)+1$.

So, we can find
$H_{1}\subset\left\{ 1,2,\ldots,k_{1}\right\} $, with $|H_{1}|=d^{m-1}(d-1)+1$
such that, $d$ divides $\sum_{t\in G_{1}}f_{1}(t)$ for all $G_{1}\subset H_{1}$
with $|G_{1}|=d$.

Using the same technique, we can find $H_{2}\subset H_{1}$,
with $|H_{2}|=d^{m-2}(d-1)+1$, such that $d$ divide $\sum_{t\in G_{2}}f_{2}(t)$
and $\sum_{t\in G_{2}}f_{1}(t)$ for all $G_{2}\subset H_{2}$ with
$|G_{2}|=d$ . Continuing, this process we can find $K_{1}\subset\left\{ 0,1,\ldots,k_{1}=d^{m}(d-1)+1\right\} $
such that, for each $f\in F$,  $d$ divides $\sum_{t\in K_{1}}f(t)$.

 Now consider the set $\left\{ k_{1}+1,k_{1}+2,\ldots,2k_{1}\right\} $ instead of $\left\{ 1,2,\ldots,k_{1}\right\}$ and by the same above argument   we can find $K_{2}\subset\left\{ k_{1}+1,k_{2}+2,\ldots,2k_{2}\right\} $
such that , for each $f\in F$ , $d$ divides $\sum_{t\in K_{2}}f(t)$. Obviously $\max K_{1}<\min K_{2}$.

Continuing the process, we get for each $n\in\mathbb{N}$, $\max K_{n}<\min K_{n+1}$
and for each $f\in F$, $d$ divides $\sum_{t\in K_{n}}f(t)$ with
$$
K_{n}\subset\left\{ (n-1)k+1,(n-1)k+2,\ldots,nk\right\} 
$$
 where $k=k_{1}=d^{|F|}(d-1)+1$.
\end{proof}

Now we state the definition of $CR$-set in \cite{HHST} by  N. Hindman, H. Hosseini, D. Strauss and M. Tootkaboni that is equivalent to $CR$-set defined by V. Bergelson and G. Glasscock in \cite[Definition 2.8]{BG} for commutative semigroups.

\begin{definition} Let $\left(S,+\right)$ be a commutative semigroup and $A\subseteq S$. Let  $k\in\mathbb{N}$.
\begin{itemize}
    \item[(a)](\textbf{$k$-$CR$-set}) A is a $k$-$CR$-set if and only if there
exists $r\in\mathbb{N}$ such that whenever $F\in\mathcal{P}_{f}\left({}^{\mathbb{N}}S\right)$,
with $|F|\le k$, there exist $a\in S$ and $H\in\mathcal{P}_{f}\left(\left\{ 1,2,\ldots,r\right\} \right)$
such that for all $f\in F$, $$a+\sum_{t\in H}f(t)\in A.$$ 
\item[(b)] A is a $CR$-set if and only if for each $k\in\mathbb{N}$,
it is a $k$-$CR$-set.

\end{itemize}

\end{definition}
 The technique of proof of the following theorem is the almost same as \cite[Theorem 3.7]{HS23}.
\begin{theorem}
    Let $m,u,v\in\mathbb{N}$, let $A$ be a $u\times v$ matrix with rational
entries and let $B$ be a $mu$-$CR$-set in $\mathbb{Z}$. Assume that
for each $a\in\mathbb{Z}$, there exists $\vec{x}\in\mathbb{Z}^{v}$
such that $A\vec{x}=\overline{a}\in\mathbb{Z}^{u}$. Then
$\left\{ \vec{y}\in\mathbb{Z}^{v}:A\vec{y}\in B^{u}\right\} $
is a $m$-$CR$-set in $\mathbb{Z}^{v}$. 
\end{theorem}

\begin{proof}
Let $C=\left\{ \vec{y}\in\mathbb{Z}^{v}:A\vec{y}\in B^{u}\right\} $,  let $F\in\mathcal{P}_{f}\left({}^{\mathbb{N}}{\left(\mathbb{Z}^{v}\right)}\right)$
with $|F|=m$. Then for any $\vec{f}\in F$ and each $t\in\mathbb{N}$,
$$
\begin{array}{ccc}
\vec{f}\left(t\right) & = & \left(\begin{array}{c}
f_{1}\left(t\right)\\
f_{2}\left(t\right)\\
\vdots\\
f_{v}\left(t\right)
\end{array}\right)\end{array}
$$

Pick $d\in\mathbb{N}$ such that all entries of $dA$ are integers.
Let 
$$
L=\left\{ \Pi_{i}\left(\vec{f}\left(t\right)\right):i\in\left\{ 1,2,\ldots,v\right\} \,\text{and}\,\vec{f}\in F\right\} 
$$
 and $|L|=mv$. From the  Lemma \ref{main lemma}, we get for each $n\in\mathbb{N}$,
$\max K_{n}<\min K_{n+1}$ and for each $f\in L$ , $d$ divides $\sum_{t\in K_{n}}f(t)$
with 
$$
K_{n}\subset\left\{ (n-1)k+1,(n-1)k+2,\ldots,nk\right\} 
$$
where $k=d^{mv}(d-1)+1$. For $\vec{f}\in F$ and each $i\in\left\{ 1,2,\ldots,u\right\} $,
and $n\in\mathbb{N}$. As $B$ is a $mu$-$CR$-set, let 
$$
g_{\vec{f},i}\left(n\right)=\sum_{j=1}^{v}a_{i,j}\sum_{t\in H_{n}}f_{j}\left(t\right)
$$
where $A=\left(a_{i,j}\right)_{u\times v}$.
Then each $g_{\vec{f},i}\left(n\right)\in{}^{\mathbb{N}}\mathbb{Z}$
and $|\left\{ g_{\vec{f},i}:i=1,2,\ldots,u\,\text{and}\,\vec{f}\in F\right\} \mid\leq mu$.
Pick $a\in\mathbb{Z}$ and $G\subset\left\{ 1,2,\ldots,r\right\} $
such that for each $\vec{f}\in F$ and each $i\in\left\{ 1,2,\ldots,u\right\} $,
$a+\sum_{n\in G}g_{\vec{f},i}\left(n\right)\in B$. 

Pick $\vec{x}\in\mathbb{Z}^{v}$ such that $A\vec{x}=\overline{a}\in\mathbb{Z}^{u}$.
Let $K=\cup_{n\in G}K_{n}$. Then 
$$
K\subset\cup_{n\in\left\{ 1,2,\ldots,r\right\} }K_{n}\subset\left\{ 1,2,\ldots,rd^{mv}\left(d-1\right)+r\right\}.
$$
We claim that for $\vec{f}\in F$, with $|F|\leq m$,  $\vec{x}+\sum_{t\in K}\vec{f}\left(t\right)\in C$.
To see this we need to show that for $i\in\left\{ 1,.2,\ldots,u\right\} $,
entry $i$ of $A\left(\vec{x}+\sum_{t\in K}\vec{f}\left(t\right)\right)$
is in $B$. That entry is 

$$
	\begin{aligned}
		a+\sum_{j=1}^{v}a_{i,j}\sum_{t\in K}f_{j}\left(t\right) & =  a+\sum_{j=1}^{v}a_{i,j}\sum_{n\in G}\sum_{t\in K_{n}}f_{j}\left(t\right)\\
		&=  a+\sum_{n\in G}\sum_{j=1}^{v}a_{i,j}\sum_{t\in K_{n}}f_{j}\left(t\right)\\
  &=a+\sum_{n\in G}g_{\vec{f},i}\left(n\right)\in B.
	\end{aligned}$$
\end{proof}
From the above theorem and fact that a set is a $CR$-set if and only
if for each $k\in\mathbb{N}$, it is a $k$-$CR$-set, we immediately
get the following corollary:

\begin{corollary}
   Let $u,v\in\mathbb{N}$, let $A$ be a $u\times v$ matrix with rational
entries and let $B$ be a $CR$-set in $\mathbb{Z}$. Assume that
for each $a\in\mathbb{Z}$, there exists $\vec{x}\in\mathbb{Z}^{v}$
such that $A\vec{x}=\overline{a}\in\mathbb{Z}^{u}$. Then
$\left\{ \vec{y}\in\mathbb{Z}^{v}:A\vec{y}\in B^{u}\right\} $
is a $CR$-set in $\mathbb{Z}^{v}$. 
\end{corollary}

\begin{theorem}\label{elementary charac}\cite[Theorem 5]{DDG}
Let $(S,+)$ be a countable commutative semigroup, then the
following are equivalent.
\begin{itemize}
\item[(a)]  $A$ is an essential $\mathcal{CR}$-set.

\item[(b)] There is a decreasing sequence $\left\langle C_{n}\right\rangle _{n=1}^{\infty}$
of subsets of $A$ such that 
\begin{itemize}
\item[(i)] for each $n\in\mathbb{N}$ and each $x\in C_{n}$, there
exists $m\in\mathbb{N}$ with $C_{m}\subseteq -x+C_{n}$ and 

\item[(ii)]  for each $n\in\mathbb{N}$, $C_{n}$ is a $\mathcal{CR}$-set.
\end{itemize}
\end{itemize}
\end{theorem}

\begin{theorem}
    Let $u,v\in\mathbb{N}$, let $A$ be a $u\times v$ matrix with rational
entries. If for any $CR$-set $B$, $\left\{ \vec{y}\in\mathbb{Z}^{v}:A\vec{y}\in B^{u}\right\} $
is $CR$-set in $\mathbb{Z}^{v}$, then for an essential $CR$-set, C,
$\left\{ \vec{y}\in\mathbb{Z}^{v}:A\vec{y}\in C^{u}\right\} $
is essential $CR$-set in $\mathbb{Z}^{v}$. 
\end{theorem}

\begin{proof}
   As $C$ is essential $CR$-set, by Theorem \ref{elementary charac}, there is a decreasing sequence $\langle C_{n}\rangle_{n=1}^{\infty}$ of subsets of $A$ such that 
   \begin{itemize}
\item[(i)] for each $n\in\mathbb{N}$ and each $\vec{x}\in C_{n}$, there
exists $m\in\mathbb{N}$ with $C_{m}\subseteq -\Vec{x}+C_{n}$ and 

\item[(ii)]  for each $n\in\mathbb{N}$, $C_{n}$ is a $CR$-set.
\end{itemize}
Now, we can construct an decreasing sequence $\langle D_{n}\rangle_{n=1}^{\infty}$, where $D_{n}=\left\{ \vec{y}\in\mathbb{Z}^{v}:A\vec{y}\in C_{n}^{u}\right\}$. As for each $n\in\mathbb{N}$, $D_{n}$ is a $CR$-set.  To prove $\left\{ \vec{y}\in\mathbb{Z}^{v}:A\vec{y}\in C^{u}\right\} $
is essential $CR$-set in $\mathbb{Z}^{v}$, it is sufficient to show that for each $n\in\mathbb{N}$ and each $x\in D_{n}$, there
exists $m\in\mathbb{N}$ with $C_{m}\subseteq -\Vec{y}+D_{n}$.

Now choose $n\in\mathbb{N}$ and $\vec{y}\in D_{n}$. Let $A=\left(a_{i,j}\right)_{u\times v}$. Then $\left\{\sum_{j=1}^{v}a_{i,j}y_{j}\in C_{n}:i\in\left\{ 1,2,\ldots,u\right\}\right\}\subset D_{n}$. For $i=1,2,\ldots,u$, there exist $m_{i}\in\mathbb{N}$ such that $C_{m_{i}}\subseteq -\sum_{j=1}^{v}a_{i,j}y_{j}+C_{n}$ and taking $m=\max\left\{m_{i}:i=1,2,\ldots,u \right\}$, we get $C_{m}\subseteq\bigcap_{i=1}^{u}\left\{\sum_{j=1}^{v}a_{i,j}y_{j}+C_{n}\right\}$. For any
$$
\vec{z}\in D_{m}\implies\left\{\sum_{j=1}^{v}a_{i,j}z_{j}:i\in\left\{ 1,2,\ldots,u\right\}\right\}\subseteq C_{m}\subseteq\bigcap_{i=1}^{u}\left\{-\sum_{j=1}^{v}a_{i,j}y_{j}+C_{n}\right\},
$$
which implies $\left\{\sum_{j=1}^{v}a_{i,j}\left(y_{j}+z_{j}\right):i\in\left\{ 1,2,\ldots,u\right\}\right\}\subset C_{n}$. Hence
$$
	\begin{aligned}
		\vec{y}+\vec{z}\in D_{n} & \implies  \vec{z}\in -\vec{y}+D_{n}\\
		&\implies D_{m}\subseteq -\vec{y}+D_{n}.
	\end{aligned}$$

\end{proof}

\bibliographystyle{plain}

\end{document}